\documentclass[12pt,preprint]{amsart}

\usepackage{amssymb,amsfonts,amsthm,amsmath,amscd}

\theoremstyle{plain}
\newtheorem{theorem}{Theorem}

\newtheorem{proposition}{Proposition}

\newtheorem{corollary}{Corollary}
\newtheorem{conjecture}{Conjecture}

\theoremstyle{remark}

\def\G{\mathcal{G}}
\def\P{\mathcal{P}}
\def\N{\mathcal{N}}

\def\X{\mathcal{X}}
\def\Y{\mathcal{Y}}

\begin{document}

\title{The game Max-Welter}

\author{Nhan Bao Ho}

\address{Department of Mathematics and Statistics, La Trobe University, Melbourne, Australia 3086}
\email{nhan.ho@latrobe.edu.au, nhanbaoho@gmail.com}

\maketitle

\begin{abstract}
\begin{sloppypar}
On a semi-infinite strip of squares rightward numbered $0, 1, 2, \ldots$ with at most one coin in each square, in Welter's game, two players alternately move a coin to an empty square on its left. Jumping over other coins is legal. The player who first cannot move  loses. We examine a variant of Welter's game, that we call Max-Welter, in which players are allowed to move only the coin furthest to the right. We solve the winning strategy and describe the positions of Sprague-Grundy value 1. We propose two theorems classifying some special cases where calculating the Sprague-Grundy value of a position of size $k$ becomes easier by considering another position of size $k-1$. We establish two results on the periodicity of the Sprague-Grundy values. We then show that the game Max-Welter is classified in a proper subclass of tame games that Gurvich calls strongly miserable.
\end{sloppypar}

\end{abstract}

\section{Introduction}
Welter's game is played with a finite number of coins on a semi-infinite strip of squares labeled $0, 1, 2, 3, \ldots$ from its left end. Each square contains at most one coin. Two players move alternately, choosing one coin and moving it to one empty square on the left. Jumping over other coins is legal. For example, in Figure \ref{W.position} (where coins are represented by the bullet symbol $\bullet$), the coin in square 6 can be moved to one of the squares 0, 1, 3, and 4. The player first unable to move loses.

\begin{figure}[ht]
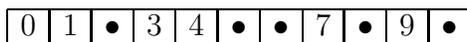

\begin{center}
\begin{tabular}{|c|c|c|c|c|c|c|c|c|c|c|}
\hline
0&1&$\bullet$&3 &4 &$\bullet$ &$\bullet$ &7 &$\bullet$ &9 &$\bullet$\\
\hline
\end{tabular}
\caption{A position in Welter's game}\label{W.position}
\end{center}
\end{figure}

As the game is finite, there is exactly one player who can win. A position is an \emph{$\N$-position} if the player, who is about to move in this position, has a strategy to win the game regardless of the opponent's strategy. Otherwise, the position is a \emph{$\P$-position}. Recall that the Sprague-Grundy value of a position $p$, denoted by $\G(p)$, is defined recursively as follows: the terminal (final) position has value 0; furthermore,  $\G(p) = n$ if and only if for every $m$ such that $0 \leq m < n$, there exists a move from $p$ to some $q$ such that $\G(q) = m$ and there is no move from $p$ to $q$ such that $\G(q) = n$.

Welter's game and its solution are generally discussed in \cite{Welter1} and \cite{Welter2}. A theory of Welter's game is analyzed in \cite[Chapter 13]{con}. A Welter's position with $k$ coins in the squares $a_1,a_2,\ldots,a_k$ in which $a_1 < a_2 < \ldots < a_k$ is denoted by $(a_1, a_2, \ldots, a_k)$. The Sprague-Grundy function for Welter's game is called the Welter function and is computed by the \emph{Mating Method} as follows. The Sprague-Grundy value of a position $(a_1,a_2,\ldots,a_k)$ is denoted by $[a_1|a_2|\ldots|a_k]$. The \emph{Nim addition}, denoted by $\oplus$, is addition in the binary number system without carrying. First mate any two $a_{i_1}, a_{i_2}$ such that $a_{i_1}, a_{i_2}$ are congruent to each other modulo the highest possible power of 2. Next mate another pair $(a_{i_3}, a_{i_4})$ from the remaining $k-2$ numbers in the same way and so on. The process ends when all numbers are mated, if $k$ is even, or all but one, $a_{i_k}$, when $k$ is odd. In the latter case $a_{i_k}$ is called a \emph{spinster}. Then,
\[ [a_1|a_2|\ldots|a_k] = [a_{i_1}|a_{i_2}] \oplus [a_{i_3}|a_{i_4}] \oplus \cdots \oplus s \quad \text{if $s$ is a spinster}\]
in which $[a|b] = (a \oplus b) - 1$. For example, the Sprague-Grundy value of the position $(2,5,6,8,10)$ in Figure \ref{W.position} is computed as follows. The mates are (2,10), (6,8) and 5 is a spinster. Therefore,
\begin{align*}
 [2|5|6|8|10] & = [2|10] \oplus [6|8] \oplus [5] \\
              & = ((2 \oplus 10)-1 ) \oplus ((6 \oplus 8)-1 ) \oplus 5 \\
              & = (8-1) \oplus (14-1) \oplus 5 = 15.
\end{align*}

Although the method of computing the Welter function, as well as a next winning move from a position, if this exists, are already known \cite[Chapter 13]{con}, there are still many interesting questions worthy of study. For example, we do not yet known the characterization of $\P$-positions with more than four coins. A different direction is to consider variants of Welter's game. A variant of Welter's game, \emph{k-Welter}, is studied in \cite{k-Welter2, k-Welter1} in which a coin can be moved to at most $k$ squares from its present square.

In this paper, we introduce a restriction of Welter's game, that we call Max-Welter, obtained as follows: from a position $(a_1,a_2,\ldots,a_k)$, one can move only the coin in the largest square $a_k$ to an empty square on its left. Jumping is legal. For example, from the position as in Figure \ref{W.position}, in Max-Welter, there are only six legal move: moving the coin in the square 10 to any of empty squares 0, 1, 3, 4, 7, and 9.

The paper is organized as follows. We describe the $\P$-positions (whose Sprague-Grundy value is $0$) in Section 2 and the positions of the Sprague-Grundy value $1$ in Section 3. Then we prove some properties of the positions of higher Sprague-Grundy values in Section 4. After this, in Section 5, we give two theorems on the Sprague-Grundy function so that calculating the Sprague-Grundy value of a position can be simplified by eliminating some small entries. We then establish two results together with three conjectures on the periodicity of Sprague-Grundy values in Section 6. In Section 7, we examine the mis\`{e}re version (in which the player who makes the last move loses) of the game Max-Welter and show that the game Max-Welter is not only tame (in the sense of \cite{con}) but also strongly miserable (in the sense of \cite{Gur1}).

\medskip
\section{The winning strategy}

This section details the winning strategy for the game Max-Welter.

\smallskip
\begin{theorem} \label{Max-P}
A position $(a_1,a_2,\ldots,a_k)$ with $k \geq 2$ is a $\P$-position if and only if $a_k=a_{k-1}+1$ and $a_{k-1}+k$ is even.
\end{theorem}

\begin{proof}
Let $\X$ be the set of positions $(a_1,a_2,\ldots,a_k)$ such that $k \geq 2$, $a_k=a_{k-1}+1$, and $a_{k-1}+k$ is even. We must prove two facts:
\begin{enumerate}
\item[1.] there is no move between any two positions in $\X$;
\item[2.] from every position $A$ not in $\X$, there is a move that terminates in $\X$,

\end{enumerate}

1. Let $(a_1,a_2,\dots,a_k)$ be a position in $\X$. Then $a_{k-1}+k$ is even. Moving the coin from the square $a_k$ results in $(a_1,a_2,\dots,a_k)$ moving to some position $(b_1,b_2,\dots,b_k)$, where $b_k = a_{k-1}$. Since $b_k-1+k = a_{k-1}-1+k$, $b_k-1+k$ is odd and so $(b_1,b_2,\dots,b_k) \notin \X$.

2. Let $(a_1,a_2,\dots,a_k)$ be a position not in $\X$. Denote this position by $A$. We first consider the case $a_k = a_{k-1}+1$.  Since $A \notin \X$, $a_{k-1}+k$ is odd. If $a_{k-2} = a_{k-1}-1$, as $A$ is not a terminal position, there exists an empty square $j$ such that $j < a_{k-2}$. Moving the coin from the square $a_k$ to the square $j$ results in $A$ moving to some position in $\X$. If $a_{k-2} < a_{k-1}-1$, moving the coin from the square $a_k$ to the square $a_{k-1}-1$ results in $A$ moving to some position in $\X$.

Consider the case $a_k > a_{k-1}+1$. Assume that moving the coin from the square $a_k$ to the square $a_{k-1}+1$ does not result in $A$ moving to some position in $\X$. Then $(a_1,a_2,\ldots,a_{k-1})$ is not a terminal position. Moreover, $a_{k-1}+k$ is odd. If $a_{k-2} = a_{k-1}-1$, as $(a_1,a_2,\ldots,a_{k-1})$ is not a terminal position, there exists an empty square $j$ such that $j < a_{k-2}$. Moving the coin from the square $a_k$ to the square $j$ results in $A$ moving to some position in $\X$. If $a_{k-2} < a_{k-1}-1$, moving the coin from the square $a_k$ to the square $a_{k-1}-1$ results in $A$ moving to some position in $\X$.
\end{proof}


\section{Positions of value 1}
Here we describe the positions of the Sprague-Grundy value $1$. They appear to be closely related to the $\P$-positions.

\smallskip
\begin{theorem}\label{Max-V1}
A position $(a_1,a_2,\ldots,a_k)$ with $k \geq 2$ has the Sprague-Grundy value $1$ if and only if one of the following conditions holds:
\begin{enumerate}
\item [$($a$)$] $(a_1,a_2,\ldots,a_k) = (0,1,\ldots,l,l+2,l+3,\ldots,k)$ for some $l$ such that $l \leq k-2$;
\item [$($b$)$] $a_k=a_{k-1}+1$ and $a_{k-1}+k$ is odd.
\end{enumerate}
\end{theorem}

\begin{proof}
Let $\Y$ be the set of positions verifying condition (a) or (b) of Theorem \ref{Max-V1}. Note that $\X$ (defined in Theorem \ref{Max-P}) is the set of $\P$-positions. We must prove two facts:
\begin{enumerate}
\item[1.] there is no move between any two positions in $\Y$,
\item[2.] from every position $A$ not in $\X \cup \Y$, there is a move that terminates in $\Y$.
\end{enumerate}

1. Let $(a_1,a_2,\dots,a_k)$ be a position in $\Y$. If $(a_1,a_2,\dots,a_k)$ is of the form (a) then the claim holds. If $(a_1,a_2,\dots,a_k)$ is of the form (b), moving the coin from the square $a_k$ leads to some position $(b_1,b_2,\dots,b_k)$ with $b_k = a_{k-1}$. Since $b_k-1+k = a_{k-1}-1+k$, $b_k-1+k$ is even and so $(b_1,b_2,\dots,b_k) \notin \Y$.

2. Let $(a_1,a_2,\dots,a_k)$ be a position not in $\X \cup \Y$. Denote this position by $A$. Note that $a_k \geq a_{k-1}+2$. Assume that $(a_1,a_2,\dots,a_{k-1})$ is a terminal position. Then $a_k \geq a_{k-1}+3$ by the condition (a). Moving the coin from the square $a_k$ to the square $a_{k-1}+2$ leads to a position in $\Y$. Assume that $(a_1,a_2,\dots,a_{k-1})$ is not a terminal position. Then there exists an empty square $j$ such that $j < a_{k-1}$. Consider $a_{k-1}+k$. If $a_{k-1}+k$ is odd, then moving the coin from the square $a_k$ to the square $a_{k-1}+1$ leads to a position in $\Y$. If $a_{k-1}+k$ is even, then $a_{k-1}-1+k$ is odd. We consider the square $a_{k-1}-1$. If the square $a_{k-1}-1$ is empty, moving the coin from the square $a_k$ to the square $a_{k-1}-1$ leads to a position in $\Y$. Otherwise, moving the coin from the square $a_k$ to the empty square $j$ leads to a position in $\Y$.
\end{proof}

\smallskip
\begin{corollary} \label{SG-V}
Let $(a_1,a_2,\ldots,a_k)$ be a position such that $k \geq 3, i \geq 0$, and $(a_1,a_2,\ldots,a_k) \neq (0,1,\ldots,k-2,k+i)$. If $a_{k-2}+1 = a_{k-1}$ and $a_{k-1} \leq a_k-2$, then $\G(a_1,a_2,\ldots,a_k) = a_k - a_{k-1}$.
\end{corollary}
\begin{proof}
Using Theorems \ref{Max-P} and \ref{Max-V1}, one can prove the corollary by induction on $a_k - a_{k-1}$.
\end{proof}

\section{A property of positions of value 2}
\begin{sloppypar}
It has been shown in Section 2 that those positions of the form $(a_1,a_2,\ldots,a_k)$ whose Sprague-Grundy values are 0 have the same difference $a_k - a_{k-1}$. The similar property also holds for the positions of the Sprague-Grundy value $1$, except for the positions of the form $(0,1,\ldots,k-2,k)$. It is natural to ask whether this circumstance is repeated for those positions whose Sprague-Grundy values are larger than 1. We now show that this connection still holds, although to a lower extent, for those positions whose Sprague-Grundy values are 2. For higher Sprague-Grundy values, this connection dies. For example, the two positions $(10,17,19)$ and $(11,12,15)$ both have Sprague-Grundy value $3$ but have different values for $a_k - a_{k-1}$.
\end{sloppypar}

\smallskip
\begin{proposition}
Let $(a_1,a_2,\ldots,a_k)$ be a position such that $k \geq 3$, $i \geq 0$, and $(a_1,a_2,\ldots,a_k) \neq (0,1,\ldots,k-2,k+i)$. If $\G(a_1,a_2,\ldots,a_k) = 2$, then $a_k - a_{k-1} = 2$.
\end{proposition}

\begin{proof}
Set $A = (a_1,a_2,\ldots,a_k)$ . By Theorems \ref{Max-P} and \ref{Max-V1}, $a_k - a_{k-1} \geq 2$. If $a_{k-2}+1=a_{k-1}$ then the proposition is true by Corollary \ref{SG-V}. We now assume that $a_{k-2}+1 < a_{k-1}$. Set $B = (a_1,a_2,\ldots,a_{k-2},a_{k-1},a_{k-1}+2)$. Note that one can move from $B$ to $(a_1,a_2,\ldots,a_{k-2},a_{k-1}-1,a_{k-1})$ and $(a_1,a_2,\ldots,a_{k-2},a_{k-1},a_{k-1}+1)$. By Theorems \ref{Max-P} and \ref{Max-V1}, we have
\begin{align*}
\{\G(a_1,a_2,\ldots,a_{k-2},a_{k-1}-1,a_{k-1}),\G(a_1,a_2,\ldots,a_{k-2},a_{k-1},a_{k-1}+1)&\} \\
                                              = \{0,1&\}
\end{align*}
and so $\G(B) \geq 2$. If $\G(B) > 2$, then there is one move from $B$ to some position $C$ whose Sprague-Grundy value is 2. Note that position $C$ can be reached from $A$ and so $\G(A) > 2$, giving a contradiction. Therefore, $\G(B) = 2$ giving $A = B$, as otherwise there exists a move from $A$ to $B$.
\end{proof}

\medskip
\section{When can the computation of the Sprague-Grundy function be simplified?}
We establish two results by which calculating the Sprague-Grundy value of a position can be simplified by reducing the input. We first consider those positions whose small empty squares can be removed without changing the game.

\smallskip
\begin{theorem} \label{ReduceSize}
Consider the position $(a_1,a_2,\ldots,a_k)$ with $k \geq 3$. If $a_1 \leq k-1$ then $\G(a_1,a_2,\ldots,a_k) = \G(a_2-1,a_3-1,\ldots,a_k-1)$.
\end{theorem}

\begin{proof}
Note that if $a_1 \leq k-1$ then the coin $a_1$ cannot be moved during the game. In fact, the coin $a_1$ can be moved if and only if there is an empty square $i$ such that $i < a_1$ even when all other coins are moved to squares on the left of the square $a_1$. This means there are at least $k$ squares on the left of the square $a_1$. This requires $a_1 \geq k$, giving a contradiction. Thus, if $a_1 \leq k-1$ then playing Max-Welter from the position $(a_1,a_2,\ldots,a_k)$ is equivalent to playing this game from the position $(a_2-1,a_3-1,\ldots,a_k-1)$ which is obtained from $(a_1,a_2,\ldots,a_k)$ by removing the square $a_1$.
\end{proof}

The next result is derived from Theorems \ref{Max-P} and \ref{Max-V1}.

\smallskip
\begin{theorem} \label{ReduceSize0}
Consider the position $(a_1,a_2,\ldots,a_k)$ with $k \geq 3$. If there exists an integer $i$ such that $a_i \geq i$ and $a_i+1=a_{i+1}$, then
\begin{align} \label{redu}
\G(a_1,a_2,\ldots,a_k) = \G(b_1,b_2, \ldots,b_j,a_i,a_{i+1},\ldots,a_k)
\end{align}
for all $j$-tuples $(b_1,b_2, \ldots,b_j)$ , where $j < a_i$ and $j+i-1$ is even.

\end{theorem}

\begin{proof}
Set $A = (a_1,a_2,\ldots,a_k)$. We prove the theorem by induction on $k-i$. As $i-1$ and $j$ have the same parity, (\ref{redu}) holds when $k-i = 1$ by Theorems \ref{Max-P} and \ref{Max-V1}. Assume that (\ref{redu}) holds for $k-i$ such that $1\leq k-i \leq n$. We show that (\ref{redu}) holds when $k-i = n+1$ by induction on $a_k - a_i$. If $a_{k-1}+1=a_k$, then (\ref{redu}) holds by Theorems \ref{Max-P} and \ref{Max-V1}. Therefore, we can assume that $a_{k - 1}+1 < a_k$. Note that $a_k-a_i = k-i$ if and only if there is no empty square between the two squares $a_i$ and $a_k$. Hence $a_k - a_i \geq k-i+1$.

First, consider the case $a_k - a_i = k-i+1$. Note that the square $a_{k - 1}+1$ is the only empty square between the two squares $a_i$ and $a_k$. Moreover, $a_{k - 2}+1 = a_{k-1}$. Let $s$ be the number of empty squares on the left of the square $a_i$ in the position $A$. Then $s = a_i-i+1 \geq 1$. Denote these $s$ empty squares by $a'_1,a'_2,\ldots, a'_s$. For each $j$ such that $1 \leq j \leq s$, denote by $A_j$ the position obtained from $A$ by moving the coin from the square $a_k$ to the square $a'_j$. Set $B = (b_1,b_2, \ldots,b_j,a_i,a_{i+1}, \ldots,a_k)$. Let $t$ be the number of empty squares on the left of the square $a_i$ in the position $B$. Then $t = a_i - j \geq 1$. Denote these $t$ empty squares by $b'_1,b'_2,\ldots, b'_t$. For each $j$ such that $1 \leq j \leq t$, denote by $B_j$ the position obtained from $B$ by moving the coin from the square $a_k$ to the square $b'_j$. We have
\begin{align*}
\begin{cases}
\G(A) = \operatorname{mex}\{\G(A_j), &\G(a_1,a_2,\ldots, a_{k-1},a_{k-1}+1) | 1 \leq j \leq s\},\\
\G(B) = \operatorname{mex}\{\G(B_l), &\G(b_1,b_2, \ldots,b_j,a_i,a_{i+1},\ldots, a_{k-1},a_{k-1}+1) \\
                     & | 1 \leq l \leq t\}.
\end{cases}
\end{align*}
By Theorems \ref{Max-P} and \ref{Max-V1},
\[\G(a_1,a_2,\ldots, a_{k-1},a_{k-1}+1) = \G(b_1,b_2, \ldots,b_j,a_i,a_{i+1},\ldots, a_{k-1},a_{k-1}+1)\]
and $\G(A_j) = \G(B_l)$ for all $j$ and $l$. Therefore, (\ref{redu}) holds for $a_k - a_i = k-i+1$.

{\sloppy
Assume that (\ref{redu}) holds when $a_k - a_i \leq m$ for some $m$ such that ${m \geq k-i+1}$. We show that (\ref{redu}) holds when $a_k - a_i = m+1$. Note that there are $a_k-a_i-k+i$ empty squares between the two squares $a_i$ and $a_k$. Also note that ${a_k-a_i-k+i=m-n \geq 2}$. Denote these empty squares by $c_1,c_2, \ldots,c_{m-n}$. For each $h$ such that $1 \leq h \leq m-n$, denote by $A_h'$ (resp.~$B_h'$) the position obtained from $A$ (resp.~$B$) by moving the coin from the square $a_k$ to the square $c_h$. With positions $A_j$ and $B_l$, where $1 \leq j \leq s$ and $1 \leq l \leq t$, defined as above. We have}
\begin{align*}
\begin{cases}
\G(A) &= \operatorname{mex}\{ \G(A_j), \G(A_t') | 1 \leq j \leq s, 1 \leq t \leq m-n\},\\
\G(B) &= \operatorname{mex}\{ \G(B_l), \G(B_t') | 0 \leq l \leq a_i-1, 1 \leq t \leq m-n\}.
\end{cases}
\end{align*}
By the inductive hypothesis on $a_k-a_i$ for the case $k-i = n+1$, we have $\G(A_t') = \G(B_t')$ for all $t$ such that $1 \leq t \leq m-n$. By the inductive hypothesis on $k-i$, we have $\G(A_j) = \G(B_l)$ for all $j, l$. Therefore, $\G(A) = \G(B)$. This completes the proof.

\end{proof}

\smallskip
\begin{corollary}
Consider the position $(a_1,a_2,\ldots,a_k)$ with $k \geq 3$. If there exists an odd integer $i$ such that $a_i \geq i$ and $a_i+1=a_{i+1}$, then
\begin{align*}
\G(a_1,a_2,\ldots,a_k) = \G(a_i,a_{i+1}, \ldots,a_k).
\end{align*}
\end{corollary}

\section{On the periodicity of the Sprague-Grundy values}

In this section, we state two theorems on the periodicity of Sprague-Grundy values. Recall that a sequence $\{s_i\}$ is \emph{periodic} if there exist integers $p$ and $n_0$ such that $s_{n + p} = s_n$, where $n \geq n_0$. The smallest such number $p$ is the \emph{period}. The sequence is \emph{additive periodic} if there exist $p$ and $n_0$ such that $s_{n + p} = s_n+p$, where $n \geq n_0$. The following theorem describes a simple additive periodicity of Sprague-Grundy values when the right end coin in a position is translated rightwards.

\smallskip
\begin{theorem} \label{Add Period}
Let $k \geq 2$ and let $a_1,a_2,\ldots,a_{k-1},a_k$ be nonnegative integers such that $a_1 < a_2 < \cdots < a_k$. There exists a positive integer $n$ such that $n \leq a_k$ and
\[
\G(a_1,a_2,\ldots,a_{k-1},a_k+n+i) = a_k+i, \quad \forall i \geq 0.
\]
\end{theorem}

\begin{proof}
First, we prove that there exists a positive integer $n$ such that $n \leq a_k$ and
\[
\G(a_1,a_2,\ldots,a_{k-1},a_k+n) = a_k.
\]
Consider the sequence
\begin{align*}
(s): \quad \G(a_1,a_2,\ldots,a_{k-1},a_k), \G(a_1,a_2,\ldots,a_{k-1},a_k+1), \ldots\\
             \ldots, \G(a_1,a_2,\ldots,a_{k-1},a_k+a_k).
\end{align*}
\begin{sloppypar} 
We claim that this sequence contains $a_k$. Since there is a move from $(a_1,a_2,\ldots,a_{k-1},a_k+j)$ to $(a_1,a_2,\ldots,a_{k-1},a_k+i)$ if $i < j$, the sequence $(s)$ contains $a_k+1$ pairwise distinct values. Assume by contradiction that the sequence $(s)$ does not contain $a_k$. Then it must contain some $m$ such that $m > a_k$. Assume that $\G(a_1,a_2,\ldots,a_{k-1},a_k+j) = m$ for some $j$ such that $j \leq a_k$. Then there exists a move from $(a_1,a_2,\ldots,a_{k-1},a_k+j)$ to some position $B$ whose Sprague-Grundy value is $a_k$. Moreover, since the sequence $(s)$ does not contain $a_k$, $B$ is of the form $(b_1,b_2,\ldots,b_k)$, where $b_k \leq a_k-1$. Note that there are fewer than $a_k-1$ empty squares on the left of the square $b_k$ in the position $B$ and so $\G(B) \leq a_k-1$, giving a contradiction. Therefore,
\[
\G(a_1,a_2,\ldots,a_{k-1},l) = a_k
\]
for some $l$ such that $a_k \leq l \leq a_k+a_k$. We now prove the existence of $n$. Since there are fewer than $a_k-1$ empty squares on the left of the square $a_k$, $\G(a_1,a_2,\ldots,a_{k-1},a_k) < a_k$ and so $l > a_k$. Let $n = l-a_k$. We have $n$ as required.
\end{sloppypar}

Next, we show that
\[
\G(a_1,a_2,\ldots,a_{k-1},a_k+n+j+1) = \G(a_1,a_2,\ldots,a_{k-1},a_k+n+j)+1
\]
for all $j$ such that $j \geq 0$. Assume by contradiction that there exists nonnegative integer $j$ such that
\[
\G(a_1,a_2,\ldots,a_{k-1},a_k+n+j+1) \geq \G(a_1,a_2,\ldots,a_{k-1},a_k+n+j)+2.
\]
Set $r = \G(a_1,a_2,\ldots,a_{k-1},a_k+n+j)$. Then there exists one move from $(a_1,a_2,\ldots,a_{k-1},a_k+n+j+1)$ to some position $C$ whose Sprague-Grundy value is $r+1$. One can see that every position reached from $C$, including $C$, can be reached from $(a_1,a_2,\ldots,a_{k-1},a_k+n+j)$ and so
\[\G(a_1,a_2,\ldots,a_{k-1},a_k+n+j) > \G(C) = r+1,\]
giving a contradiction. This completes the proof.
\end{proof}

We next consider how Sprague-Grundy values vary when all coins in a position are rightwards translated. The following theorem shows that when positions have two neighbouring coins, the Sprague-Grundy values are invariant under rightward translation.

\smallskip
\begin{theorem} \label{Period p=1}

Let $k \geq 3$ and let $a_1,a_2,\ldots,a_{k-1},a_k$ be nonnegative integers such that $a_1 < a_2 <\ldots < a_k$ and $a_k > a_{k-1}+1$. If there exists $i$ such that $i \leq k-2$, $a_i \geq k-2$, and $a_i+1=a_{i+1}$, then
\begin{align} \label{i,i+1}
\G(a_1+1,a_2+1,\ldots,a_k+1) = \G(a_1,a_2,\ldots,a_k).
\end{align}
\end{theorem}

\begin{proof}

We prove (\ref{i,i+1}) by induction on $k$. We first show that (\ref{i,i+1}) holds for $k = 3$ by induction on $a_k-a_i$ with $a_k-a_i \geq 3$. Note that $i = 1$ and $k = 3$ in this case. Assume that $a_k - a_i = 3$. We have
\begin{align*}
\G(a_1,a_2,a_3) = \operatorname{mex}\{ \G(i,a_1,a_2), \G(a_1,a_2,a_2+1) | 0 \leq i \leq a_1-1\}
\end{align*}
and
\begin{align*}
\G(a_1+1,&a_2+1,a_3+1) = \\
         &\operatorname{mex}\{ \G(j,a_1+1,a_2+1), \G(a_1+1,a_2+1,a_2+2) | 0 \leq j \leq a_1\}.
\end{align*}
By Theorems \ref{Max-P} and \ref{Max-V1}, we have $\G(a_1+1,a_2+1,a_2+2) = \G(i,a_1,a_2)$ for all $i$ and $\G(j,a_1+1,a_2+1) = \G(a_1,a_2,a_2+1)$ for all $j$ and so
\[\G(a_1+1,a_2+1,a_3+1) = \G(a_1,a_2,a_3).\]
Assume that, for $k = 3$, (\ref{i,i+1}) holds when $3 \leq a_k-a_i \leq n$. We show that (\ref{i,i+1}) holds when $a_k-a_i = n+1$. We have

\begin{align*}
\G(a_1,&a_2,a_3) = \\
           &\operatorname{mex}\{ \G(i,a_1,a_2), \G(a_1,a_2,s) | 0 \leq i \leq a_1-1, a_2+1 \leq s \leq a_3-1\}
\end{align*}
and
\begin{align*}
&\G(a_1+1,a_2+1,a_3+1) = \\
&\operatorname{mex}\{\G(j,a_1+1,a_2+1), \G(a_1+1,a_2+1,t) | 0 \leq j \leq a_1, a_2+2 \leq t \leq a_3\}.
\end{align*}
For $s = a_2+1$, by Theorems \ref{Max-P} and \ref{Max-V1}, we have
\[\G(j,a_1+1,a_2+1) = \G(a_1,a_2,s), \quad \forall j.\]
For $t = a_2+2$, by Theorems \ref{Max-P} and \ref{Max-V1}, we have
\[\G(a_1+1,a_2+1,t) = \G(i,a_1,a_2), \quad \forall i.\]
By inductive hypothesis, we have $\G(a_1+1,a_2+1,t) = \G(a_1,a_2,t-1)$, where $a_2+3 \leq t \leq a_3$, and so
\[\{\G(a_1+1,a_2+1,t) | a_2+3 \leq t \leq a_3\} = \{\G(a_1,a_2,s) | a_2+2 \leq s \leq a_3-1\}.\]
Therefore,
\[\G(a_1+1,a_2+1,a_3+1) = \G(a_1,a_2,a_3).\]

Let $m \geq 3$ and assume that (\ref{i,i+1}) holds for $k$ with $3 \leq k \leq m$. We show that (\ref{i,i+1}) holds for $k = m+1$. The proof for this case is essentially the same as that of the case $k = 3$ and we leave the details to the reader. By the principle of induction, (\ref{i,i+1}) holds for all $k \geq 3$.

\end{proof}

We end this section with three conjectures on the periodicity of Sprague-Grundy values, based on computations we have performed.


We call the {\it 2-coin $k$-sequence} the sequence $\{(a_n,b_n)\}_{n \geq 1}$ of all positions of the Sprague-Grundy value 2 in which $a_i < a_j$ if $i < j$. We have the following conjecture on the sequence members of $\{(a_n,b_n)\}_{n \geq 1}$ when $n$ is large enough.

\begin{conjecture} \label{2coin}
Let $\{(a_n,b_n)\}_{n \geq 1}$ be the 2-coin $k$-sequence of Max-Welter. When $n$ is large enough, we have
\begin{align} \label{2c}
b_n - a_n = k+1-\lfloor\frac{k+1}{2}\rfloor,
\end{align}
which is a constant.
\end{conjecture}

Our computation shows that the convergence of $b_n - a_n$ in (\ref{2c}) comes quite early. For example, up to $k = 10$, the convergence happens after 50.

We now consider the sequence $\{\G(a+i,b+i)\}_{i \geq 0}$, where $a < b$. For each $i$, $(a+i,b+i)$ belongs to some $k_i$-sequence. When $i$ is large enough, Conjecture \ref{2coin}, if held, implies that $2(b-a)-2 \leq k_i \leq 2(b-a)-1$. We now give a more general conjecture.

\smallskip
\begin{conjecture} \label{Obs0}
Let $(a_1,a_2,\ldots,a_k)$ be a position with $k \geq 2$. The sequence
\[\{\G(a_1,a_2,\ldots,a_{k-2},a_{k-1}+i,a_k+i)\}_{i \geq 0}\]
is ultimately periodic with the period $2(a_k-a_{k-1})$. Moreover, the periodic sequence members are bounded by $2(a_k - a_{k-1})-2$ and $2(a_k - a_{k-1})-1$.
\end{conjecture}

Conjecture \ref{Obs0} represents the periodicity of the Sprague-Grundy values of those positions obtained by rightwards translating the last two coins from a position. The following is our observation on the periodicity of the Sprague-Grundy values of those positions obtained by rightwards translating the last any number of coins from a position.

\smallskip
\begin{conjecture} \label{Obs3}
Let $(a_1,a_2,\ldots,a_k)$ be a position such that $k \geq 3$, $a_k > a_{k-1}+1$. Let $l \leq k-2$. Consider the sequence
\[(s') = \{\G(a_1,a_2,\ldots,a_{l-1},a_l+i,\ldots,a_k+i)\}_{i \geq 0}.\]
If there exist distinct $i,j$ such that $i \geq l$, $j \geq l$, and $a_{i+1}-a_i \neq a_{j+1}-a_j$, then the sequence $(s')$
is ultimately periodic with period $p=1$. Otherwise, the sequence is ultimately periodic with $p = 2(a_k - a_{k-1})$.
\end{conjecture}

\section{The mis\`{e}re version}
This section examines the mis\`{e}re version of the game Max-Welter. We show that the game Max-Welter belongs to a proper subclass of the so-called tame games. We first revise some background of the mis\`{e}re play.

Recall that a game is under \emph{mis\`{e}re} (resp,~\emph{normal}) convention if the player who makes the last move is declared to be the loser (resp.~winner). Mis\`{e}re games have been intensively studied and exposed many notable properties; see \cite{Gur1} for reference of papers on mis\`{e}re games.

For each game $G$ played under normal convention, denote by $G^-$ the mis\`{e}re version of $G$. Define by $\G$ and $\G^-$ the Sprague-Grundy functions of the games $G$ and $G^-$ respectively. The later function is computed as follows: the terminal position has value 1 and the mis\`{e}re Sprague-Grundy value $\G^-(p)$ of a non-terminal position $p$ is the smallest nonnegative integer which is different to $\G^-(q)$ for any $q$ such that there exists a move from $p$ to $q$. Also note that a position $p$ in the mis\`{e}re game $G^-$ is second player winning if and only if $\G^-(p) = 0$.

We say that the two functions $\G$ and $\G^-$ {\em swap} values 0 and 1 if for every position $p$, we have $(\G(p),\G^-(p)) \in \{(0,1), (1,0), (k,k) | k \geq 2\}$. We first show that the game Max-Welter and its mis\`{e}re form swap values 0 and 1. The proof of the following two theorems are essentially the same as those of Theorems \ref{Max-P} and \ref{Max-V1} respectively. We leave the details to the reader.

\smallskip
\begin{theorem}\label{MMax-P}
A position $(a_1,a_2,\ldots,a_k)$, where $k \geq 2$, is a $\P$-position in mis\`{e}re Max-Welter if and only if one of the following holds:
\begin{enumerate}
\item [$($a$)$] $(a_1,a_2,\ldots,a_k) = (0,1,\ldots,l,l+2,l+3,\ldots,k)$ for some $l \leq k-2$,
\item [$($b$)$] $a_k=a_{k-1}+1$ and $a_{k-1}+k$ is odd.
\end{enumerate}
\end{theorem}

\smallskip
\begin{theorem} \label{MMax-V1}
A position $(a_1,a_2,\ldots,a_k)$ with $k \geq 2$ has mis\`{e}re Sprague-Grundy value $1$ if and only if $a_k=a_{k-1}+1$ and $a_{k-1}+k$ is even.
\end{theorem}

A game is {\em tame} (first defined by J.H. Conway in \cite{con}) if in every position $p$, either $\G(p) = \G^-(p)$ or $\G(p) + \G^-(p) = 1$. The classic and generalized versions of Nim \cite{con, Gur1} provide examples. In \cite{Gur1}, V.A. Gurvich introduces a proper subclass of tame games that he calls {\em strongly miserable} games. These are games whose $\G$ and $\G^-$ functions (i) swap values 0 and 1, and (ii) agree on other values. Note that there is no position $p$ in a strongly miserable such that either $\G(p) = \G^-(p) = 0$ or $\G(p) = \G^-(p) = 1$. The game of Nim is tame but not strongly miserable \cite{Gur1}. Examples of strongly miserable games can be found in \cite{Gur1}; they include the so-called subtraction games \cite{Ann-mis, sub-mis}, as well as some generalizations of Nim analyzed in \cite{Wyt-mis}.

In particular, Gurvich \cite[Theorem 3]{Gur1} establishes several sufficient conditions for a game to be strongly miserable. One of these conditions is that $\P$-positions of the normal and mis\`{e}re version are disjoint. By Gurvich's result and Theorems \ref{MMax-P} and \ref{MMax-V1}, the game Max-Welter is strongly miserable. We now give a direct proof of this result, independent of Gurvich's work (for those who would like a straight approach).

\begin{theorem} \label{SMGe}
The game Max-Welter is strongly miserable
\end{theorem}

\begin{proof}

Note that a game can be described as a finite directed acyclic graph without multiple edges where each vertex is a position and each downward edge between vertices is a move. A source is a vertex with no incoming edges and it is equivalent to the original position of the game. A sink is a vertex with no outgoing edges, so a sink of the graph is a terminal position of the game.

Let $G$ be the graph of the game max-Welter. We define the \emph{height} of a position $p$, denoted by $h(p)$, the length of the longest directed path from $p$ to the sink. We prove by induction on the height $h(p)$ of position $p$ that ``if $\G(p) \geq 2$ then $\G^-(p) = \G(p)$" $(*)$. Note that if $h(p) = 1$ then $\G^-(p) = 0$ and $\G(p) = 1$. Assume that $(*)$ is true when $h(p) \leq n$. We show that $(*)$ is true when $h(p) = n+1$.

We can assume that $\G(p) = m$ for some $m \geq 2$. For each $k$ such that $k < m$, there exists $w_k$ such that $\G(w_k) = k$ and one can move from $p$ to $w_k$. Note that $h(w_k) \leq n$. By the swapping on Sprague-Grundy values 0 and 1 of $\G$ and $\G^-$, and by the inductive hypothesis, we have
\begin{align*}
\begin{cases}
\G^-(w_1) = \G(w_0) = 0 \text{ and } \G^-(w_0) = \G(w_1) = 1, \\
\G^-(w_k) = \G(w_k) = k \text{ for } 2 \leq k < m.
\end{cases}
\end{align*}
Therefore,
\[\{\G^-(w_k)| 0 \leq k < m\} = \{\G(w_k)| 0 \leq k < m\} = \{0,1, \ldots, m-1\}.\]
Note that if there exists one move from $p$ to some $w$ in $G^-$ then this move can also be made in $G$. Note that $\G(w) \neq m$. Also note that $h(w) \leq n$. By the swap on values 0 and 1 of $\G$ and $\G^-$, and by the inductive hypothesis, either $\G(w) + \G^-(w) = 1$ or $\G^-(w) = \G(w) \geq 2$ and so $\G^-(w) \neq m$. We have
\[\G^-(p) = \operatorname{mex}\{\G(w) | \text{ if there exists a move from $p$ to $w$}\}.\]
Since the $\operatorname{mex}$ set includes the set $\{0,1, \ldots, m-1\}$ and excludes $m$, ${\G^-(p) = m}$. This completes the proof.
\end{proof}

\section{Concluding remarks}
In this paper we studied the game Max-Welter.  We not only solved this game but also characterized $P$-positions and positions of the Sprague-Grundy value $1$. We exposed some properties related to symmetry and, more importantly, periodicity of this game. Finally, it is proven that the game Max-Welter is strongly miserable.

\section*{Acknowledgements}
I am indebted to the two referees for their helpful, detailed, constructive, and informative comments.


\end{document}